\title{Left and right convergence of graphs with bounded degree}
\author{Christian Borgs, Jennifer Chayes,
Jeff Kahn\footnote{Research supported by NSF grant DMS0701175},
L\'aszl\'o Lov\'asz\footnote{Research supported by OTKA grant
No.~67867}}
\date{(January 31, 2010)}

\documentclass[fleqn]{article}
\usepackage{graphicx}
\usepackage{amssymb,amsmath}
\usepackage{theorem,hyperref}

\newtheorem{theorem}{Theorem}[section]
\newtheorem{prop}[theorem]{Proposition}
\newtheorem{lemma}[theorem]{Lemma}

\newtheorem{corollary}[theorem]{Corollary}
\theorembodyfont{\rmfamily}
\newtheorem{remark}{Remark}

\newenvironment{proof}{\medskip\noindent{\bf Proof. }}{\hfill$\square$\medskip}
\newenvironment{proof*}[1]{\medskip\noindent{\bf Proof of #1.}}{\hfill$\square$\medskip}
\begin{document}

\addtolength{\baselineskip}{3pt} \setlength{\oddsidemargin}{0.2in}

\def\col{{\rm col}}
\def\hom{{\rm hom}}
\def\Hom{{\rm Hom}}
\def\Inj{{\rm Inj}}
\def\Ind{{\rm Ind}}
\def\inj{{\rm inj}}
\def\ind{{\rm ind}}
\def\aut{{\rm aut}}
\def\Surj{{\rm Surj}}
\def\surj{{\rm surj}}
\def\Aut{{\rm Aut}}
\def\PAG{{\rm PAG}}
\def\Match{{\rm Match}}
\def\ConIn{{\rm CInd}}
\def\CSpan{{\rm CSpan}}
\def\SpT{{\rm SpTr}}
\def\Conn{{\rm Con}}
\def\Sub{{\rm Sub}}

\def\chr{{\sf chr}}
\def\stab{{\sf stab}}
\def\bet{{\sf bet}}
\def\tree{{\sf tree}}
\def\csp{{\sf cr}}

\long\def\killtext#1{}

\long\def\privremark#1{\medskip[{\it {#1}}]\medskip}

\def\proofend{\hfill$\square$\medskip}
\def\Proof{\noindent{\bf Proof. }}

\def\ra{\rightarrow}
\def\mn{\medskip\noindent}
\def\gc{\gamma}
\def\gs{\sigma}
\def\gb{\beta}
\def\ga{\alpha}
\def\gz{\zeta}
\def\gl{\lambda}
\def\gL{\Lambda}
\def\sm{\setminus}
\def\sub{\subseteq}
\def\ent{{\sf H}}
\def\R{{\mathbb R}}
\def\one{{\mathbf 1}}
\def\Q{{\mathbb Q}}
\def\Z{{\mathbb Z}}
\def\C{{\mathbb C}}

\def\AA{{\cal A}}\def\BB{{\cal B}}\def\CC{{\cal C}}\def\DD{{\cal D}}\def\EE{{\cal E}}\def\FF{{\cal F}}
\def\GG{{\cal G}}\def\HH{{\cal H}}\def\II{{\cal I}}\def\JJ{{\cal J}}\def\KK{{\cal K}}\def\LL{{\cal L}}
\def\MM{{\cal M}}\def\NN{{\cal N}}\def\OO{{\cal O}}\def\PP{{\cal P}}\def\QQ{{\cal Q}}\def\RR{{\cal R}}
\def\SS{{\cal S}}\def\TT{{\cal T}}\def\UU{{\cal U}}\def\VV{{\cal V}}\def\WW{{\cal W}}\def\XX{{\cal X}}
\def\YY{{\cal Y}}\def\ZZ{{\cal Z}}
\def\P{{\sf P}}
\def\E{{\sf E}}
\def\Var{{\sf Var}}
\def\T{^{\sf T}}

\def\tr{{\rm tr}}
\def\cost{\hbox{\rm cost}}
\def\val{\hbox{\rm val}}
\def\rk{{\rm rk}}
\def\diam{{\rm diam}}
\def\Ker{{\rm Ker}}
\def\eps{\varepsilon}
\def\supp{{\rm supp}}

\maketitle

\tableofcontents

\begin{abstract}
The theory of convergent graph sequences has been worked out in two
extreme cases, dense graphs and bounded degree graphs. One can define
convergence in terms of counting homomorphisms from fixed graphs into
members of the sequence (left-convergence), or counting homomorphisms
into fixed graphs (right-convergence). Under appropriate conditions,
these two ways of defining convergence was proved to be equivalent in
the dense case by Borgs, Chayes, Lov\'asz, S\'os and Vesztergombi. In
this paper a similar equivalence is established in the bounded degree
case.

In terms of statistical physics, the implication that left
convergence implies right convergence means that for a
left-convergent sequence, partition functions of a large class of
statistical physics models converge. The proof relies on techniques
from statistical physics, like cluster expansion and Dobrushin
Uniqueness.
\end{abstract}

\section{Introduction}

The theory of convergent graph sequences has been worked out in two
extreme cases, dense graphs and bounded degree graphs. The case of
dense graphs is probably easier; convergence of such graphs was
introduced and characterized in different ways in
\cite{BCLSV1,BCLSV2}. Convergence of bounded degree graph sequences
was defined by Benjamini and Schramm \cite{BSch}, and has inspired a
lot of work \cite{BSchSh,Ly,Elek1,Schramm1}. While this is perhaps
the more important case from the point of view of applications, the
theory is less complete, and, for example, some of the
characterizations of convergence in the dense case have no analogues.
The goal of this paper is to prove, for bounded degree graphs, an
analogue of one of these characterizations.

For a simple finite graph $G$ and node $v\in V(G)$, let $B_G(v,r)$
denote the subgraph of $G$ induced by the nodes at distance at most
$r$ from $v$. We consider this as a rooted graph, with $v$ designated
as its root. { Following \cite{BSch},} a sequence $(G_1,\dots,G_n)$
of simple graphs with degrees uniformly bounded by $D$ is called {\it
locally convergent} (weakly convergent, left-convergent) if for every
fixed integer $r\ge 1$, selecting a random node { $v$} uniformly from
$V(G)$, the probability that $B_{G_n}(v,r)$ is a fixed rooted graph
$U$ tends to a limit as $n\to\infty$. It is not hard to see that this
definition is equivalent to saying that for every connected graph
$F$, $\hom(F,G_n)/|G_n|$ tends to a limit, where $\hom(F,G_n)$ is the
number of homomorphisms (adjacency preserving maps) of $F$ into
$G_n$, and $|G_n|=|V(G_n)|$ is the number of nodes of $G_n$.

This reformulation raises the possibility of defining convergence by
turning the arrows around, i.e., in terms of $\hom(G_n,H)$ for fixed
graphs $H$. It is easy to see that the ``right'' normalization in
this case is $\ln\hom(G_n,H)/|G_n|$.

In general, convergence of $\ln\hom(G_n,H)/|G_n|$ does not follow
from left-convergence. As an example, if $G_n$ is the $n$-cycle, then
the sequence $(G_n)$ is trivially locally convergent. But if $K_2$ is
the complete graph with 2 nodes, then $\ln\hom(G_n,{ K_2})/|G_n|$
alternates between $(\ln 2)/n$ and $-\infty$ depending on the parity
of $n$.

We will prove (Theorem \ref{THM:HOM-CONV} that if $H$ is sufficiently
dense (depending on $D$), then local convergence does imply that
$\ln\hom(G_n,H)/|G_n|$ is convergent. Conversely
(\ref{THM:RIGHT2LEFT}), if $\ln\hom(G_n,H)/|G_n|$ is convergent for
every sufficiently dense $H$, then $(G_n)$ is locally convergent.

\section{Preliminaries}

A graph is {\it simple} if it has no loops or parallel edges. A {\it
simple looped graph} is a simple graph with a loop added at each of
its node. A {\it weighted graph} $H$ is a graph with a weight
$\alpha_i(H)>0$ associated with each node $i$ and a real weight
$\beta_{ij}(H)$ associated with each edge $ij$. (The weights
$\beta_{ij}(H)$ may be negative.) We will consider each weighted
graph as a complete graph with loops at all nodes; absence of an edge
is indicated by weight $0$. (No parallel edges are allowed in a
weighted graph.) If the graph $H$ is understood from the context, we
use the notation $\alpha_i$ and $\beta_{ij}$. We can consider every
simple graph (looped or not looped) as a weighted graph on the same
node set, where the original edges have weight $1$ and the missing
edges have weight $0$. Let $\alpha_H=\sum_{i\in V(H)}\alpha_i(H)$
denote the total nodeweight of $H$.

We denote by $\overline{H}$ the weighted graph obtained from $H$ by
replacing each edgeweight $\beta_{ij}$ by $1-\beta_{ij}$. (Note that
this also applies to loops. So if $H$ is the weighted graph
corresponding to a simple graph, then $\overline{H}$ corresponds to
the complement of $H$ with a loop added at each node.)

Let $\Sub(G)$ denote the set of nonempty subgraphs of $G$ without
isolated nodes, $\Conn(G)$, the set of connected subgraphs of $G$
with at least two nodes, $\ConIn(G)$, the set of connected induced
subgraphs of $G$ with at least two nodes, and $\CSpan(G)$, the set of
connected spanning subgraphs of $G$ { (with one or more nodes).}

To simplify notation, we will write $|G|=|V(G)|$.

Let $\{A_1,\dots,A_n\}$ be a family of sets. We denote by
$L(A_1,\dots,A_n)$ their intersection graph, i.e., the graph on
$V=[n]$ in which we connect $i$ and $j$ by an edge if $A_i\cap
A_j\not=\emptyset$.

For a set $V$, we denote by $\Pi(V)$ the set of its partitions.

\subsection{Homomorphism numbers and densities}

For two (finite) simple graphs $F$ and $G$, $\hom(F,G)$ denotes the
number of homomorphisms (adjacency preserving maps) from $F$ to $G$.
We define $t(F,G)$ to be the probability that a random map of $V(F)$
into $V(G)$ is a homomorphism, i.e.,
\[
t(F,G)=\frac{\hom(F,G)}{|G|^{|F|}}.
\]
We call $t(F,G)$ the {\it homomorphism density} of $F$ in $G$.

Sometimes we need to consider the number of injective homomorphisms
$\inj(F,G)$, the number of embeddings as induced subgraphs,
$\ind(F,G)$, and the number of homomorphisms surjective on both the
nodes and the edges, $\surj(F,G)$. We denote by
$\aut(F)=\ind(F,F)=\surj(F,F)$ the number of automorphisms of $F$.
The quotients $\inj_0(F,G)= \inj(F,G)/\inj(F,F)$ and $\ind_0(F,G)=
\ind(F,G)/\ind(F,F)$ count the numbers of subgraphs and induced
subgraphs of $G$ isomorphic to $F$, respectively.

We extend these notions to the case when the target graph (denoted by
$H$ in this case) is weighted:
\[
\hom(F,H)=\sum_{\phi:~V(F)\to V(H)} \prod_{u\in V(F)}
\alpha_{\phi(u)}(H) \prod_{uv\in E(F)} \beta_{\phi(u),\phi(v)}(H)
\]
where the sum runs over all maps from $V(F)$ to $V(HG)$. The
homomorphism density is now defined as
\[
t(F,H)=\frac{\hom(F,H)}{{\alpha_H}^{|F|}}.
\]
We will also need the following related quantity:
\begin{equation}\label{zGH}
z(G,H)= \sum_{F\in\CSpan(G)} (-1)^{|E(F)|} t(F,H).
\end{equation}
This quantity plays an important role in the sequel, and it would be
interesting to explore its combinatorial significance. We have the
inverse relation
\begin{align}\label{t2z}
\sum_{F\in\CSpan(G)} &(-1)^{|E(F)|} z(F,H) = \sum_{F\in\CSpan(G)}
(-1)^{|E(F)|} \sum_{J\in\CSpan(F)} (-1)^{|E(J)|} t(J,H)\nonumber\\
&= \sum_{J\in\CSpan(G)} t(J,H) \sum_{J\subseteq F\subseteq G}
(-1)^{|E(F)|-|E(J)|}  = t(G,H).
\end{align}

\subsection{Local convergence of a graph sequence}

Let $G=(V,E)$ be a graph with degrees bounded by $D$, and fix an
integer $r\ge 0$. For $v\in V$, let $N(v)=N_G(v)$ denote set of nodes
adjacent to $v$. For $v\in V$, let $B(v,r)=B_G(v,r)$ denote the
subgraph of $G$ induced by all nodes at a distance at most $r$ from
$v$ (the {\it $r$-neighborhood} of node $v$, or the {\it $r$-ball}
about $v$). So $B(v,1)$ is the subgraph induced by $N(v)\cup\{v\}$.
We consider $B(v,r)$ as a rooted graph, i.e., the node $v$ is
specified as the center of $B(v,r)$. For fixed $r$, there is a finite
number (depending on $D$ and $r$ only) of possible $r$-balls
$U_1,\dots,U_N$. Let $\mu(G,U_i)$ denote the fraction of nodes of $G$
whose $r$-neighborhood is $U_i$. We can think of $\mu$ as a
probability distribution on possible $r$-balls.

We call a sequence $(G_n)$ of graphs with degrees bounded by $D$ {\it
locally convergent}, or {\it left-convergent} if $\mu(G_n,U)$ tends
to a limit $\mu(U)$ as $n\to\infty$ for every $r$ and every $r$-ball
$U$.

It is easy to see that if $(G_n)$ is left-convergent, then for every
connected graph $F$ the sequence $\hom(F,G_n)/|G_n|$ is convergent.
It is also quite easy to see that this property is sufficient for
local convergence. By formulas \eqref{zGH} and \eqref{t2z}, we could
also require that $z(F,G_n)$ is convergent for every simple graph
$F$.

We can also talk about {\it right-convergence}, meaning that
$\ln\hom(G_n,H)/|G_n|$ tends to a limit as $n\to\infty$ (here $H$ can
be a weighted graph; we'll see that the normalization is
appropriate). However, we will not formally define right-convergence,
since there seem to be different ways of specifying which weighted
graphs $H$ to consider here. Rather, we want to find a reasonable
class of weighted graphs $H$ for which right-convergence is
equivalent to local (left-) convergence.

\subsection{Chromatic polynomial}

Let $G=(V,E)$ be a simple graph with $n$ nodes. For every nonnegative
integer $y$, we denote by $\chr(G,y)$ the number of $y$-colorations
of $G$. Note that $\chr(G,q)=\hom(G,K_q)$.

Let $\chr_0(G,k)$ denote the number of $k$-colorations of $G$ in
which all colors occur. Then clearly
\begin{equation}\label{CHROM-COL}
\chr(G,y)=\sum_{k=0}^\infty \chr_0(G,k)\binom{y}{k}.
\end{equation}
This implies that $\chr(G,y)$ is a polynomial in $y$ with leading
term $y^n$ and constant term $0$, which is called the {\it chromatic
polynomial} of $G$. It is easy to see that if $G$ is a simple graph,
then for every $e\in E(G)$,
\begin{equation}\label{EQ:CHROM-REC}
\chr(G,q)=\chr(G\setminus e,q)-\chr(G/e,q),
\end{equation}
where $G\setminus e$ and $G/e$ arise from $G$ by deleting and
contracting $e$, respectively (in $G/e$, parallel edges are collapsed
to one). From this recurrence a number of properties of the chromatic
polynomial are easily proved, for example, that its coefficients
alternate in sign.

The coefficient of the linear term in the chromatic polynomial is
\begin{equation}\label{EQ:PSI-SUM}
\csp(G)=\sum_{G'\in\CSpan} (-1)^{|E(G')|},
\end{equation}
which is called the {\it Crapo invariant} or {\it chromatic
invariant} of the graph. Trivially, $\csp(G)=0$ if $G$ is
disconnected. It follows from \eqref{EQ:CHROM-REC} that if $G$ is a
simple graph, then for every $e\in E(G)$,
\begin{equation}\label{EQ:BETA-REC}
\csp(G)=\csp(G\backslash e)-\csp(G/e).
\end{equation}
This implies by induction that $(-1)^{|G|-1}\csp(G)>0$ if $G$ is
connected.

\subsection{Subtree counts}

Let $\tree(G)$ denote the number of spanning trees in $G$. More
generally, let $\tree(G;v,k)$ denote the number of subtrees of $G$
with $k$ nodes, containing a given node $v\in V(G)$. Let $T_D$ be an
infinite rooted $D$-ary tree, with root $r$. The following formula is
well known (\cite{Stan}, Theorem 5.3.10):
\begin{equation}\label{EQ:TREENUM}
\tree(T_D;r,k) = \frac{1}{k}\binom{kD}{k-1}.
\end{equation}
The right hand side has this more convenient estimate:
\[
\frac{1}{k}\binom{kD}{k-1} \le \frac{(kD)^{k-1}}{k!} <
\frac{e^kD^{k-1}}{k\sqrt{2\pi k}} < \frac{(eD)^{k-1}}{2}
\]
(assuming $k\ge 3$, but the bound is trivially true for $k=2$ as
well).

\begin{lemma}\label{LEM:SUBTREECOUNT}
Let $G$ be a graph with maximum degree $D$ and let $v\in V(G)$.

\smallskip

{\rm (a)} The number of subtrees of $G$ with $k$ nodes containing $v$
is at most $\frac{1}{k}\binom{kD}{k-1}$.

\smallskip

{\rm (b)} The number of connected subgraphs of $G$ with $m$ edges
containing $v$ is at most $\frac{1}{m+1}\binom{(m+1)D}{m}$.

\smallskip

{\rm (c)} The number of connected induced subgraphs of $G$ with $k$
nodes containing $v$ is at most $\frac{1}{k}\binom{kD}{k-1}$.

\smallskip

{\rm (d)} The number of connected subgraphs of $G$ with $k$ nodes
containing $v$ is at most $2^{Dk}$.
\end{lemma}

\begin{proof}
{ Throughout this proof, we drop the subscript $D$ from $T_D$.}

(a) It suffices to note the easy fact that for any graph $G$ with
maximum degree $D$, the number of subtrees with $k$ nodes containing
$v$ is not larger than the corresponding number in $T$. Indeed, we
could replace $T$ by the rooted tree in which the root has degree
$D$, all the other nodes have (down-)degree $D-1$.

(b) Similarly as in (a) it suffices to prove that for any graph $G$
with maximum degree $D$, the number of connected subgraphs of $G$
with $m$ edges containing $v$ is not larger than the corresponding
number in $T$.

Let us label the edges going out of any given node $v$ of $G$
arbitrarily by $1,\dots,d_G(v)$. Also, label the edges going from a
node of $T$ to its children arbitrarily $1,\dots,D$. Let $F$ be a
connected subgraph of $G$ containing $v$. Let $T_0$ be a spanning
tree in $F$. Orient $F$ so that $T_0$ is oriented away from the root
$v$ (the other edges are oriented arbitrarily).

There is a unique embedding $\phi:~V(F)\hookrightarrow V(T)$ that
preserves the edges of $F$, and for at each node $u\in V(F)$, the
edges of $T_0$ leaving $u$ in $T_0$ are mapped onto edges in $T$ with
the same label. For every edge $ab\in E(F)\setminus E(T_0)$, map it
onto the edge of $T$ leaving $\phi(a)$ with the same label. This
assigns to $F$ a subtree $F'$ of $T$ with $m$ edges.

Clearly, $F'$ uniquely determines $F$: starting from the root, we can
map the edges of $T$ back into $G$. This proves (b).

(c) is a trivial consequence of (a).

(d) We can select the edges of $F$ incident with $v$ in less than
$2^D$ ways; then going to one of the neighbors $v_1$ of $v$, we can
select the set of edges of $F$ incident with $v_1$ in less than $2^D$
ways, etc. Repeating this $k$ times, we have finished selecting $F$.
\end{proof}

\subsection{Weighted subtrees and weighted chromatic invariants}

Let $H$ be a weighted graph. We extend the definitions of the number
of subtrees and of the chromatic invariant to weighted graphs:
\[
\csp(H)=\sum_{F\in\CSpan(H)}  (-1)^{|E(F)|} \prod_{e\in F}\beta_e,
\qquad \tree(H)=\sum_{F\in\SpT(G)} \prod_{e\in F} \beta_e.
\]
We note that if all edgeweights of $H$ are $1$, then $\tree(H)$ is
the number of spanning trees and $\csp(H)$ is the chromatic invariant
of the underlying simple graph. (The nodeweights play no role in
these definitions.)

For $e\in E(H)$, let $H-e$ denote the weighted graph obtained from
$H$ by deleting the edge $e$. We need two versions of the operation
of contracting an edge. Let $H/e$ denote the graph obtained by
contracting $e$, where the arising parallel edges are replaced by a
single edge whose weight is the sum of the weights of its pre-images.
Let $H\div e$ denote the graph obtained from $H$ similarly, except
that the new edgeweight is the sum minus the product of the weights
of its pre-images. (Note that for graphs $H$ with edgeweights between
$0$ and $1$, the resulting edgeweight again lies between $0$ and $1$,
which is not necessarily the case for $H/ e$). The two quantities
introduced above satisfy the recurrence relations
\begin{equation}\label{EQ:CSP-REC}
\csp(H) = \csp(H-e) - \beta_e\csp(H\div e),\qquad \tree(H) =
\tree(H-e) + \beta_e \tree(H/e).
\end{equation}
For graphs with edge weights between $0$ and $1$, the first of these
relations implies that $(-1)^{|H|-1}\csp(H)>0$.

Let $G$ be a simple graph and $H$, a weighted graph, and let
$\tilde\alpha$ be the normalized weight
$\tilde\alpha_i=\alpha_i(F)/\alpha_F$. By a {\it random map} $G\to H$
we mean a map $V(G)\to V(H)$, where the image of each node of $G$ is
chosen independently from the probability distribution
$\tilde\alpha$.

For any map $\phi:~V(G)\to V(H)$, we can define a weighting of $G$,
where the weight of an edge of $G$ is the weight of its image in $H$.
We denote this weighted graph by $G^\phi$. Note that for a random map
$G\to H$ we have
\[
t(G,H)= \E_\phi \prod_{ij\in E(G)}\beta_{\phi(i)\phi(j)},
\]
and so
\begin{align}\label{EQ:CR-Z}
z(G,H) &= \sum_{F\in\CSpan(G)}(-1)^{|E(F)|} t(F,H) \nonumber\\
&= \sum_{F\in\CSpan(G)} (-1)^{|E(F)|} \E_\phi \prod_{ij\in
E(F)}\beta_{\phi(i)\phi(j)} = \E_\phi \csp(G^\phi).
\end{align}

\begin{lemma}\label{LEM:SP-TREE}
Let $H$ be a weighted graph with node weights $1$ and edge weights in
$[0,1]$, then
\[
|\csp(H)| \le \tree(H).
\]
\end{lemma}

\begin{proof}
By equation \eqref{EQ:CSP-REC},
\[
|\csp(H)| \le |\csp(H-e)| + \beta_e(H) |\csp(H\div e)|.
\]
Since the edgeweights in $H\div e$ are not larger than the
corresponding edgeweights in $H/e$, we get by induction on the number
of edges that
\[
|\csp(H)| \le \tree(H-e) + \beta_e(H) \tree(H/e) = \tree(H).
\]
\end{proof}

To state our next lemma, we define
\[
c(H)= \max_{u\in V(H)} \sum_{v\in V(H)} \frac{\alpha_v}{\alpha_H}
|\beta_{uv}|.
\]

\begin{lemma}\label{LEM:SUMTREE}
Let $G$ be a simple graph, and let $H$ be a weighted graph. Let
$\phi$ be a random map $G\to H$. Then
\[
\E_\phi|\tree(G^\phi)| \le \tree(G) c(H)^{|G|-1}.
\]
\end{lemma}

\begin{proof}
We may assume the edgeweights in $H$ are nonnegative. We have
\begin{equation}\label{EQ:EPHI}
\E_\phi(\tree(G^\phi)) = \E_\phi\Bigl(\sum_{T\in\SpT(G)} \prod_{ij\in
E(T)} \beta_{\phi(i)\phi(j)} \Bigr) = \sum_{T\in\SpT(G)}
\E_\phi\Bigl(\prod_{ij\in E(T)} \beta_{\phi(i)\phi(j)} \Bigr)
\end{equation}
Fix the tree $T$, and let $p$ be one of its endpoints, with single
neighbor $q$. Then picking the random map $\psi$ of
$V(G)\setminus\{p\}$ first and the the image $u$ of $p$ last, we get
\begin{align*}
\E_\phi\Bigl(\prod_{ij\in E(T)} \beta_{\phi(i)\phi(j)} \Bigr)&=
\E_\psi\Bigl(\prod_{ij\in E(T-p)} \beta_{\phi(i)\phi(j)}
\E_u(\beta_{\psi(q)u})\Bigr)\\
&= \E_\psi\Bigl(\prod_{ij\in E(T-p)} \beta_{\phi(i)\phi(j)} \sum_u
\frac{\alpha_u}{\alpha_H} \beta_{\psi(q)u}\Bigr)\\
&\le \E_\psi\Bigl(\prod_{ij\in E(T-p)} \beta_{\phi(i)\phi(j)}
c(H)\Bigr) = c(H) \E_\psi\Bigl(\prod_{ij\in E(T-p)}
\beta_{\phi(i)\phi(j)}\Bigr),
\end{align*}
whence by induction
\[
\E_\phi\Bigl(\prod_{ij\in E(T)}\beta_{\phi(i)\phi(j)} \Bigr)\le
c(H)^{|T|-1}.
\]
By \eqref{EQ:EPHI}, the Lemma follows.
\end{proof}

\begin{lemma}\label{LEM:ZGH1}
Let $G$ be a simple graph, and let $H$ be a weighted graph with
edge weights in $[0,1]$. Then
\[
|z(G,H)| \le \tree(G) c(H)^{|G|-1}.
\]
\end{lemma}

\begin{proof}
Let $\phi$ be a random map $G\to T$. Then by \eqref{EQ:CR-Z} and
Lemmas \ref{LEM:SP-TREE} and \ref{LEM:SUMTREE},
\[
|z(G,H)|\le \E_\phi |\csp(G^\phi)| \le \E_\phi \tree(G^\phi) \le
\tree(G) c(H)^{|G|-1}.
\]
\end{proof}

\section{Left-convergence implies right-convergence}

Our first main theorem is the following.

\begin{theorem}\label{THM:HOM-CONV}
Let $(G_n)$ be a left-convergent sequence of graphs with maximum
degree at most $D$. Let $H$ be a weighted graph with
$0\le\beta_{ij}\le 1$ and $c(\overline{H})< 1/(2D)$. Then $\ln
t(G_n,H)/|G_n|$ is convergent as $n\to\infty$.
\end{theorem}

Recall that the condition on $c(\overline{H})$ means that
\begin{equation}\label{EQ:CLOSE-1}
\sum_{k\in V(H)}\frac{\alpha_k}{\alpha_H}(1-\beta_{ik})<\frac{1}{2D}
\end{equation}
for all $i \in V(H)$. It is clear that we may assume that
$\alpha_H=1$.

We give two proofs of this Theorem: one, using the Dobrushin
Uniqueness Theorem, and another one using cluster expansion
techniques. In fact, the second proves a weaker result only, where in
\eqref{EQ:CLOSE-1}, the $2D$ in the denominator is replaced by the
stronger condition $8D$. The reason for giving it at all is that (a)
it uses a completely different technique, (b) it gives approximation
formulas for $\ln t(G,H)/|G|$ with explicit error bounds, and (c) the
method can also be used to prove the converse of the theorem (see
Theorem \ref{THM:RIGHT2LEFT}).

\subsection{Proof via Dobrushin Uniqueness}

We briefly recall two basic notions (see e.g. \cite{LPW} for a more
informative discussion):

\mn (i)  A {\em coupling} of probability distributions $\mu$ and
$\nu$ is a random pair $(X,Y)$ defined on some probability space such
that the marginal distribution of $X$ is $\mu$ and that of $Y$ is
$\nu$. A coupling of (not necessarily real-valued) random variables
$\phi,\psi$ is a random pair $(X,Y)$ such that the laws of $X$ and
$Y$ are those of $\phi$ and $\psi$ (respectively).

\mn (ii)  The {\em total variation distance} between discrete
probability distributions $\mu$ and $\nu$ on $\Omega$ is
$\|\nu-\mu\|_{_{{\rm TV}}}= \frac{1}{2}\sum_{\omega\in
\Omega}|\mu(\omega)-\nu(\omega)|$; it is equal to the minimum over
couplings $(X,Y)$ of $\mu$ and $\nu$ of $\Pr(X\neq Y)$. The total
variation distance of a pair of random variables $\phi,\psi$ is the
total variation distance of their distributions.

\medskip

For the rest of this section we fix $H$ as in Theorem
\ref{THM:HOM-CONV}, and set $t(G,H)=\hom(G,H)=t(G)$. For any $G$ and
$\phi:V(G)\rightarrow  V(H)$, set
\[
W(\phi) = \prod_{u\in V(G)} \alpha_{\phi(u)} \prod_{uv\in E(G)}
\beta_{\phi(u),\phi(v)}.
\]
The natural associated probability measure on $V(H)^{V(G)}$ is given
by $\Pr_{G}(\phi)\propto W(\phi)$ (that is,
$\Pr_{G}(\phi)=W(\phi)/t(G)$). We write $\E_{G}$ for expectation with
respect to this measure.

Given $\gL\sub V(G)$ and $\ga:V(G)\sm\gL\ra V(H)$, let
$\phi_{\ga}:V(G)\ra V(H)$ be chosen according to
\[
\mbox{$\Pr(\phi_{\ga}=\tau) =\Pr_G(\phi =\tau|\mbox{$\phi\equiv \ga $
off $\gL$}) ~~~~~~ \forall\tau:V(G)\ra V(H).$}
\]

For $\gz=(\gz_1,\dots,\gz_s)$ with $\gz_i\in V(H)$, define the
probability distribution $\nu_{\gz}$ on $V(H)$ by
\[
\nu_{\gz}(i)\propto \ga_i\prod_{j=1}^s\gb_{i,\gz_j}.
\]
(Thus $\nu_{\gz}$ is the conditional distribution of $\phi(v)$ given
that $d(v)=s$ and the $\phi$-values of the neighbors of $v$ are
$\gz_1,\dots,\gz_s$.) The following version of Dobrushin Uniqueness
is convenient for our purposes, but see e.g. \cite{Georgii} for a
more usual statement.

\begin{theorem}\label{DU}
Let $G$ be a graph with maximum degree at most $D$, and let $H$ be a
weighted graph such that with notation as above, there is a
$0<\kappa<1$ such that for any $s\leq D$ and
$\gz=(\gz_1,\dots,\gz_s)$ and $\gz'=(\gz_1,\dots,\gz_{s-1},\gz'_s)$
with $\gz_1,\dots,\gz_s,\gz_s'\in V(H)$, we have
\begin{equation}\label{nugz}
\|\nu_{\gz}-\nu_{\gz'}\|_{_{{\rm TV}}}\leq \frac\kappa{D}.
\end{equation}
Let $\gL\sub V(G)$, $\gL'=V(G)\setminus \gL$ and $\ga,\gb:~\gL'\to
V(H)$. Then there is a coupling
$(\tilde{\phi}_{\ga},\tilde{\phi}_{\gb})$ of $\phi_{\ga}$ and
$\phi_{\gb}$ such that
\[
\Pr(\tilde{\phi}_{\ga}\neq \tilde{\phi}_{\gb}) \leq
\kappa^{d(x,\gL')} ~~~~\forall x\in V(G).
\]
In particular, for any $\Omega\sub \gL$ the total variation distance
of the restrictions of $\phi_{\ga}$ and $\phi_{\gb}$ to $\Omega$ is
at most $\sum_{x\in \Omega}\kappa^{d(x,\gL')}$.
\end{theorem}

To apply this theorem, we need a couple of simple facts.

\begin{prop}\label{TVD}
Let $\gc_i\geq \mu_i,\nu_i\geq 0$ for $i=1,\dots,n$, $\gc=\sum
\gc_i$, $\mu=\sum\mu_i$, $\nu=\sum \nu_i$ and $\xi \geq 0$ and
suppose ($\gc \geq$) $\nu, \mu\geq \gc-\xi$. Then
\[
\sum_{i=1}^n \big| \frac{\mu_i}{\mu}-\frac{\nu_i}{\nu}\big| \leq
\frac{2\xi}{\gc-\xi}~;
\]
that is, the total variation distance of the distributions
$\{\mu_i/\mu\}_{i\in [n]}$ and $\{\nu_i/\nu\}_{i\in [n]}$ is at most
$\xi/(\gc-\xi)$.
\end{prop}

\begin{proof}
Assuming (w.l.o.g.) that $\nu\geq \mu$, we have (with sums over $i\in
[n]$),
\begin{align*}
\sum\big| \frac{\mu_i}{\mu}-\frac{\nu_i}{\nu}\big| &\leq
\sum\mu_i\left(\frac1\mu-\frac1\nu\right) +
\frac{1}{\nu}\sum |\mu_i-\nu_i| \\
& \leq
\frac{\nu-\mu}{\nu} +\frac{1}{\nu}\sum ((\gc_i-\mu_i)+(\gc_i-\nu_i))\\
& = 2\frac{\gc-\mu}{\nu} ~\leq ~ 2\frac{\xi}{\gc-\xi}.
\end{align*}
\end{proof}

\begin{prop}\label{PROP:NUGZ}
The conditions on $H$ in Theorem \ref{THM:HOM-CONV} imply that
(\ref{nugz}) holds with $\kappa=2Dc(\overline{H})$.
\end{prop}

\begin{proof}
Suppose $H$ is as in Theorem \ref{THM:HOM-CONV} and let $\gz,\gz'$ be
as in Theorem \ref{DU}. For $i\in V(H)$, let $ \gc_i
=\ga_i\prod_{j=1}^{s-1}\gb_{i,\gz_j}$, $\mu_i = \gc_i\gb_{i,\gz_s}$
and $\nu_i = \gc_i\gb_{i,\gz_s'}$, and set $\gc =\sum\gc_i$. Then
$\mu_i,\nu_i\leq \gc_i$ and (using the inequality $\prod \eta_i\geq
1-\sum (1-\eta_i)$ for $\eta_i\in [0,1]$) we have
\begin{equation}\label{gcsum}
\sum_i\gc_i \geq \sum_i \ga_i[1-\sum_{j=1}^{s-1}(1-\gb_{i,\gz_j})] =
1-\sum_{j=1}^{s-1}\sum_i\ga_i(1-\gb_{i,\gz_j})\geq
1-(D-1)c(\overline{H}),
\end{equation}
and
\begin{equation}\label{musum}
\sum_i\mu_i = \sum_i \gc_i[1-(1-\gb_{i,\gz_s})] \geq \gc -
\sum_i\ga_i(1-\gb_{i,\gz_s}) \geq \gc - c(\overline{H}),
\end{equation}
and similarly $\sum \nu_i\geq \gc -c(\overline{H})$; so Proposition
\ref{TVD} shows that \eqref{nugz} holds as claimed.
\end{proof}

\begin{proof*}{Theorem \ref{THM:HOM-CONV}}
Our approach here via (\ref{logtG}) is similar to that of \cite{BG},
which in turn was inspired by the ``cavity" method of statistical
physics; see e.g. \cite{MP}.

Given an ordering $v_1,\dots ,v_n$ (with $n=|G|$) of $V(G)$, set
$G_k=G-\{v_1,\dots,v_k\}$. We have
\[
t(G_k)~=\sum_{\phi:V(G_k)\ra H}W(\phi)
\]
and may write
\[
t(G_{k-1})= \sum_{\phi:V(G_k)\ra H}W(\phi) \sum_{i\in
V(H)}\alpha_i\prod_{w\in N_{G_{k-1}}(v_k)} \beta_{i,\phi(w)}.
\]
Thus
\[
\frac{t(G_{k-1})}{t(G_k)} = \E_{G_k}\sum_{i\in
V(H)}\alpha_i\prod_{w\in N_{G_{k-1}}(v_k)}\beta_{i,\phi(w)},
\]
and
\begin{equation}\label{logtG}
\ln t(G) = \sum_{k=1}^n \ln \E_{G_k}\sum_{i\in
V(H)}\alpha_i\prod_{w\in N_{G_{k-1}}(v_k)}\beta_{i,\phi(w)}.
\end{equation}
We will use Theorem \ref{DU} to say that for large $r$ the
expectation in (\ref{logtG}) is nearly determined by the
$r$-neighborhood of $v_k$ in $G_{k-1}$. To say this properly set, for
a graph $K$ and $v\in V(K)$,
\begin{equation}\label{PsiK}
\Psi_K(v) = \E_{K-v}\sum_{i\in V(H)}\ga_i\prod_{w\in
N_K(v)}\gb_{i,\phi(w)}.
\end{equation}
We note right away that
\begin{equation}\label{PSI-BOUND}
\frac12 < \sum_{i\in V(H)}\ga_i\prod_{w\in N_K(v)}\gb_{i,\phi(w)} \le
1, \qquad\text{and so} \qquad \frac12 < \Psi_K(v) \le 1.
\end{equation}
The upper bound is trivial, while the lower bound follows from a
computation similar to \eqref{gcsum}:
\begin{align*}
\sum_{i\in V(H)}\ga_i\prod_{w\in N_K(v)}\gb_{i,\phi(w)} &\ge
\sum_{i\in V(H)}\alpha_i\Bigl(1-\sum_{w\in N_K(v)}
(1-\beta_{i,\phi(w)})\Bigr)\nonumber\\
&= 1- \sum_{w\in N_K(v)} \sum_{i\in V(H)} \alpha_i
(1-\beta_{i,\phi(w)}) \ge 1-Dc(\overline{H})>\frac12.
\end{align*}

The assertion is then that for $K$ of maximum degree at most $D$,
$\Psi_K(v)$ is determined to within $o_r(1)$ by (the isomorphism type
of) $B_K(v,r)$ (where $o_r(1)\ra 0$ as $r\ra \infty$); that is:

\begin{lemma}\label{KK'L}
For any $K,K'$ of maximum degree at most $D$, $v\in V(K)$ and $v'\in
V(K')$ with $B_{K'}(v',r)\cong B_K(v,r)$, we have

\mn {\rm (a)} $|\Psi_K(v)-\Psi_{K'}(v')|< D\kappa^r$,

\mn {\rm (b)} $|\ln\Psi_K(v)-\ln \Psi_{K'}(v')|< 2D\kappa^r$.
\end{lemma}

\begin{proof}
(a) The sum in (\ref{PsiK}) is a function of the multiset $M(v,\phi)
= \{\phi(w):w\in N_K(v)\}$. By Theorem \ref{DU} there is a coupling
$(\tilde{\phi},\tilde{\phi}')$ of $\phi$ and $\phi'$ chosen according
to $\Pr_{K-v}$ and $\Pr_{K'-v'}$ so that $\Pr(M(v,\tilde{\phi})\neq
M(v',\tilde{\phi}'))\leq |N_K(v)|\kappa^r$. With this coupling, using
the upper bound in \eqref{PSI-BOUND},
\[
|\Psi_K(v)-\Psi_{K'}(v')|\leq  \Pr(M(v,\tilde{\phi})\neq
M(v',\tilde{\phi}')) \le D\kappa^r.
\]

\mn (b)  This is implied by (a) once we observe that $\Psi_K(v)$ is
bounded below by \eqref{PSI-BOUND}.
\end{proof}

Returning to the proof of Theorem \ref{THM:HOM-CONV}, it's convenient
to speak of an ordering $\gs$ of $V(G)$, thought of as a bijection
from $V(G)$ to $[n]$ (again with $n=|V(G)|$). For such a $\gs$ and
$v\in V(G)$, set $G(v,\gs) = G[\{w\in V(G):\gs(w) \geq \gs(v)\}]$.
Then with $\gs$ a random (uniform) permutation of $V(G)$,
\eqref{logtG} gives
\begin{equation}\label{Egs}
\ln t(G) = \sum_{v\in V(G)}\E_{\gs}\ln\Psi_{G(v,\gs)}(v).
\end{equation}
By Lemma \ref{KK'L} the contribution of $v$ to (\ref{Egs}) is
determined up to $o_r(1)$ by $B_G(v,r)$. Precisely, let $U=B_G(v,r)$
and $U_\sigma=B_{G(v,\gs)}(v,r)$, then
\begin{equation}\label{vcontrib}
\E_{\gs}\ln\Psi_{G(v,\gs)}(v) =  X_U+R,
\end{equation}
where $X_U=\E_{\gs} \ln \Psi_{U_\sigma}(v)$ depends on the ball
$U=B_G(v,r)$ only, and $|R|<2D\kappa^r$. By \eqref{PSI-BOUND}, we
have $|X_U|<1$.

Thus $|G|^{-1}\ln t(G)= \sum \mu(G,U)X_U+o_r(1) $ (with the sum over
$r$-balls $U$) and
\begin{equation}\label{last}
\big||G_m|^{-1}\ln t(G_m)-|G_n|^{-1}\ln t(G_n)\big| < \sum
|\mu(G_m,U)-\mu(G_n,U)| +o_r(1).
\end{equation}
Finally, the right hand side of (\ref{last}) can be made as small as
desired by choosing a sufficiently large $r$ and then $m,n$ large
enough to make the sum small.
\end{proof*}

\begin{remark}\label{1}
Of course the condition on $H$ in Theorem \ref{THM:HOM-CONV} can be
replaced by any assumption that supports the conclusions of Lemma
\ref{KK'L} (with some $o_r(1)$ in place of the explicit bounds given
there). One notable example involves the {\em hard-core model}, in
which $V(H)=\{0,1\}$ and the weights are $\ga_0=1/(1+\gl)$, $\ga_1
=\gl/(1+\gl)$, $\gb_{0,1}=\gb_{0,0}=1$ and $\gb_{1,1} = 0$.  Here the
present results combined with \cite{Weitz} give the convergence in
Theorem \ref{THM:HOM-CONV} provided $\gl\leq
(D-1)^{D-1}/(D-2)^D\approx e/D$ (whereas Theorem \ref{DU} gives this
for $\gl< 1/D$).

Another very interesting example is that of counting $q$-colorings;
thus $H$ is the complete graph on $[q]$ (without weights, though to
put it in the above framework we should replace $\ga_i=1$ by
$\ga_i=1/q~\forall i$). Here Theorem \ref{DU} gives convergence for
$q>2D$, but it seems reasonable to expect that $q\geq D+1$ suffices.
That this is at least true for large girth (that is, if we add the
requirement that the girth of $G_n$ tends to infinity), follows from
the present arguments with Theorem \ref{DU} replaced by a result of
Jonasson \cite{Jonasson} which says (informally) that for a uniform
$q$-coloring of an $r$-branching tree with $q\geq r+2$, the color of
the root becomes nearly independent of the colors of the leaves as
the depth of the tree grows. (We actually need this for trees in
which each internal node has {\em at most} $r$ children, but this
version is easily seen to follow from the original.)

If we assume, in addition to large girth, that the $G_n$ are $D$-{\em
regular}, then we have (again for $q\geq D+1$) the explicit limit
\begin{equation}\label{reggirth}
\frac{\ln \hom (G_n,H)}{|G_n|}\ra \ln q + \frac{D}2\ln(1-\frac1q).
\end{equation}
This is one of the main results of \cite{BG}, obtained there by
combining the cavity method with a ``rewiring" device (another idea
from statistical physics \cite{MP}), used to maintain regularity.
Here we have the result more easily:  it follows from the observation
that Johansson's theorem (which is also needed in \cite{BG}) implies
that the expectations in (\ref{Egs}) tend to $(D/2)\ln(1-1/q)$ as the
girth grows; namely, it implies that for each $i\in [q]$ the events
$\{\gs(w)>\gs(v) ~\mbox{and} ~\phi(w) =i\}$ ($w\in N(v)$) are, for
large girth, nearly independent, each with probability about $1/q$.
(The key difference between the present argument and that of
\cite{BG} is the use of the random ordering $\gs$.)
\end{remark}

\begin{remark}\label{2}
An argument similar to the one above gives (for a left-convergent
sequence $\{G_n\}$) convergence of $\{|G_n|^{-1}\ent(\phi_{G_n})\}$,
where $\ent$ is (say binary) entropy and $\phi_{G_n}:V(G_n)\ra H$ is
chosen according to $\Pr_{G_n}$.  Here we should replace
(\ref{logtG}) by the ``chain rule" expansion $ \ent(\phi) =
\sum_v\ent(\phi(v)|(\phi(w):\gs(w)>\gs(v))). $ Getting to the
analogue of (\ref{vcontrib}) now requires an extra step:  we should
choose $r_1$ so that for any $\gs$ the law of $\phi(v)$ given
$(\phi(w):\gs(w)>\gs(v))$ is ``nearly determined" by
$(\phi(w):\gs(w)>\gs(v), w\in B(v,r_1))$,
and then $r$ ($>r_1$) so that the law of the latter vector is nearly
unaffected by the values taken by $\phi$ outside $B(v,r)$.
\end{remark}

\subsection{Proof via Mayer expansion}

Our second proof relies on techniques which are well know in the
mathematical statistical physics literature. To apply these
techniques, we express $t(G,H)$ as the partition function of a so
called abstract polymer system, express its logarithm in terms of an
infinite series whose terms can be written down explicitly, and
finally prove that for $c(\overline{H})< 1/(8D)$, the series for
$\frac 1{|G|}\ln t(G,H)$ is absolutely convergent uniformly in $|G|$.
This will allow us to take the limit in Theorem~\ref{THM:HOM-CONV}
term by term.

\subsubsection{Stable sets, Mayer expansion, and Dobrushin's lemma}

We start with some preliminaries from mathematical physics,
reformulated here in a more combinatorial language. Let $G$ be a
graph and let $\II(G)$ denote the set of stable (independent) subsets
of $V(G)$. We assign a variable $x_i$ to each node $i$, and define
the {\it multivariate stable set polynomial} as
\[
\stab(G,\mathbf{x}) = \sum_{S\in\II(G)} \prod_{i\in S} x_i.
\]
Note that $\stab(G,1,\dots,1)=\hom(G,H)$, where $H$ is the graph on
two adjacent nodes, with a loop at one of them (all weights being
$1$).

In the language of mathematical physics, the pair $(G,\mathbf{x})$ is
called an {\it abstract polymer system}, and $\stab(G,\mathbf{x})$ is
called the {\it partition function} of the abstract polymer system
$(G,\mathbf{x})$ (see, e.g., \cite{Seiler}, where the notion of an
abstract polymer system was first introduced). Here we will be
interested in the Taylor expansion of $\ln\stab(G,\mathbf{x})$ about
$\mathbf{x}=(0,\dots,0)$, known under the name of Mayer expansion in
statistical physics.  In a slightly less general context than the one
considered here, this expansion was first derived in  Malyshev
\cite{Mal79}, who in turn relied heavily on the work of Rota
\cite{Rota}.  In the general context of an abstract polymer system,
it goes back to \cite{Seiler}.

\killtext{ For $a\in \Z_+^V$, let $a!=a_1!\cdots a_m!$. Let $G[a]$
denote the graph obtained by replacing each node $i$ of $V$ by $a_i$
twin copies and adding an edge between any pair of twins. For
$\mathbf{x}=(x_i:~i\in V)$, we define $\mathbf{x}^b=x_1^{a_1}\cdots
x_m^{a_m}$. For $I\subseteq V$, we set $\mathbf{x}^I=\prod_{i\in I}
x_i$.}

For a sequence
$v\in V^m$ of nodes of a simple graph $G$, let $G[v]$ denote the
graph on $[m]$ in which $i$ and $j$ are adjacent if and only if $v_i$
and $v_j$ are equal or adjacent in $G$. (Note that $v$ may contain
repetitions.) The following lemma is a reformulation of a result of
Seiler \cite{Seiler}.

\begin{lemma}\label{COR:LNST-A}
Let $G=(V,E)$ be a simple graph. For every $\mathbf{x}\in\R^V$ such
that the series below is absolutely convergent, we have
\begin{equation}\label{EQ:LNST-A}
\ln\stab(G,\mathbf{x}) = \sum_{m=1}^\infty\frac1{m!} \sum_{v\in V^m}
\csp(G[v]) \prod_{i=1}^m x_{v_i},
\end{equation}
\end{lemma}

To prove absolute convergence of the expansion in \eqref{EQ:LNST-A},
we use the following lemma which goes back to Dobrushin.  In the form
stated here, it can be found, e.g., in \cite{Bo}.

\begin{lemma}
\label{thm:Dobrushin}
Let $G=(V,E)$ be a simple graph, and let $\mathbf{x}\in\R^V$ and
$\mathbf{b}\in [0,\infty)^V$ be such
\begin{equation}
\label{AP.Conv-cond} \sum_{ j\in V\atop {ij\in E {~\rm or~} j=i}}
\ln\Bigl(1+|x_j|e^{b_j}\Bigr)\leq b_i.
\end{equation}
for all $i\in V$. Then the series in \eqref{EQ:LNST-A} is absolutely
convergent, and
\begin{equation}
\label{AP.logZ-Dbd} \bigl|\ln\stab(G,\mathbf{x})\bigr| \leq
\sum_{i\in V} \ln(1+|x_i|e^{b_i}).
\end{equation}
\end{lemma}

\subsubsection{Mayer expansion for $\ln t(G,H)$}

We can rewrite $\hom(G,H)$ in terms of the intersection graph
$\GG=L(\ConIn(G))$ of connected induced subgraphs.

\begin{lemma}\label{LEM:HOM-STAB}
For every simple graph $G$ and weighted graph $H$, define a vector
$z\in\R^{\ConIn(G)}$ by $z_F=z(F,\overline{H})$. Then $t(G,H) =
\stab(\GG,z)$.
\end{lemma}

\begin{proof}
By easy computation,
\begin{equation}\label{EQ:TGH-1}
t(G,H)=\sum_{E'\subseteq E} (-1)^{|E'|}t(G',\overline{H}),
\end{equation}
where $G'=(V(G),E')$. Using that $t(G',\overline{H})$ is
multiplicative over the components of $G'$ and that singleton
components give a factor of $1$, we get
\begin{equation}\label{tGH1}
t(G,H) =\sum_{E'\subseteq E}\prod_{F} (-1)^{|E(F)|}
t(F,\overline{H}),
\end{equation}
where the product extends over all non-singleton components of $G'$.
Collecting terms that induce the same partition, we get
\begin{align*}\label{}
t(G,H) &= \sum_{\PP\in\Pi(V)} \prod_{Y\in\PP}
\sum_{F\in\CSpan(G[Y])} (-1)^{|E(F)|} t(F, \overline{H})\\
&=\sum_{\PP\in \Pi(V)} \prod_{Y\in\PP}
z(G[Y],\overline{H})=\stab(\GG,z).
\end{align*}
\end{proof}

For any multiset $\{F_1,\dots,F_k\}$ of subgraphs, let
$L(F_1,\dots,F_k)$ denote the intersection graph of
$V(F_1),\dots,V(F_k)$. Combining Lemma \ref{LEM:HOM-STAB} and Lemma
\ref{COR:LNST-A}, we get the following formula.

\begin{corollary}\label{COR:EXP1}
Let $G$ be a simple graph and $H$, a weighted graph. If the series
below is absolute convergent, we have
\begin{equation}\label{EQ:LNT-FZ}
\ln t(G,H) = \sum_{m=1}^\infty \frac1{m!}\sum_{F_1,\dots,F_m\in
\ConIn(G)} \csp(L(F_1,\dots,F_m)) \prod_{j=1}^m z(F_j,\overline{H}).
\end{equation}
\end{corollary}

Next we establish the convergence condition \eqref{AP.Conv-cond} for
$\mathbf b$ of the form $b_F=b|F|$.  For vectors $\mathbf b$ of this
form, it is clearly enough to prove that for all $i\in V$, we have
\begin{equation}
\label{Conv-cond} \sum_{{ F\in \ConIn(G):\atop i\in V(F)}}
\ln\Bigl(1+|z_F| e^{b|F|}\Bigr)\leq b.
\end{equation}
We in fact prove a stronger inequality. To state our result, we
define
\begin{equation}\label{Kb}
K= \frac{b+e^b}{\ln(1+be^{-b})}, \qquad \epsilon=-\ln(DK
c(\overline{H})) \qquad\text{and}\qquad
\tilde z_F=e^{\epsilon (|F|-1)}z_F.
\end{equation}
We will assume that $c(H)<1/(DK)$, so that $\eps>0$.

In the special case of colorings, i.e., the case where $H$ is the
complete graph without loops, the next lemma was already shown in
\cite{Bo}.

\begin{lemma}
\label{lem:conv-cond} For every simple graph $G$ with maximum degree
$D$, every weighted graph $H$ with edge weights in $[0,1]$, and every
node $i\in V(G)$, we have
\begin{equation}
\label{Conv-cond-1} \sum_{{ F\in \ConIn(G) \atop V(F)\ni i}}
|\tilde{z}_F| e^{b|F|} \leq b.
\end{equation}
\end{lemma}

\begin{remark}\label{}
The lemma clearly implies condition \eqref{Conv-cond}. In fact,
\eqref{Conv-cond} holds even if $z$ is replaced by $\tilde z$.
\end{remark}

\begin{proof}
Using the bound in Lemma \ref{LEM:ZGH1}, it is enough to show that
\begin{equation}
\label{Conv-cond-3} \sum_{W\subseteq V:\atop i\in W,~ |W|\geq 2}
\tree(G[W])\lambda ^{|W|-1}\leq b e^{-b}.
\end{equation}
where $\lambda=e^b/(KD)$. Consider a tree $T$ contributing to $\tree
(G[W])$. After removing the point $i$ from $T$, the tree $T$
decomposes into a certain number of connected components $T_1,\dots,
T_k$, with vertex sets $U_1,\dots,U_k$. Note that
$\Pi\{U_1,\dots,U_k\}$ is a partition of $W\setminus\{i\}$ into
disjoint subsets. Classifying the spanning trees of $G[W]$ according
to these partitions, one easily obtains the identity
\begin{equation}
\label{Tree-ident} \tree(G[W])=\sum_{\Pi}\prod_{U\in \Pi}
\left(\tree(G[U])\sum_{j\in U\atop ij\in E}1\right),
\end{equation}
where the sum runs over partitions of $W\setminus\{i\}$ into disjoint
subsets.  With the help of this identity, one easily bounds the left
hand side of \eqref{Conv-cond-3} by induction on the number of
vertices in $V$.  Indeed, we first rewrite the left hand side as
\begin{equation}
\label{Conv-cond-4}
\begin{aligned}
\sum_{W\subseteq V:\atop {i\in W,~|W|\geq 2}}
& \tree(G[W])\lambda^{|W|-1}
\\ &= \sum_{W\subseteq V:\atop {i\in W,~|W|\geq 2}}
\sum_{k=1}^D\frac 1{k!} \sum_{U_1,\dots,U_k\subseteq W\setminus\{i\}
\atop {W\setminus\{i\}=\bigcup_s U_s \atop U_s\cap
U_r=\emptyset\text{ for
}s\neq r}} \prod_{s=1}^k \left(\tree(G[U_s])\lambda^{|U_s|}\sum_{j\in
U_s\atop ij\in E}1\right)
\\
&= \sum_{k=1}^D\frac 1{k!}
\sum_{U_1,\dots,U_k\subseteq V\setminus\{i\} \atop U_s\cap
U_r=\emptyset\text{ for }s\neq r} \prod_{s=1}^k
\left(\tree(G[U_s])\lambda^{|U_s|}\sum_{j\in U_s\atop ij\in
E}1\right)
\\
&{ = \sum_{k=1}^D\frac 1{k!} \sum_{j_1,\dots,j_k\in N(i)\atop j_r\neq
j_s\text{ for }s\neq r} \sum_{U_1,\dots,U_k\subseteq V\setminus\{i\}
\atop {U_s\cap U_r=\emptyset\text{ for }s\neq r\atop j_s\in U_s\text
{ for all }s}} \prod_{s=1}^k \tree(G[U_s])\lambda^{|U_s|}.}
\end{aligned}
\end{equation}
{ Rewriting the first two sums as a sum over subsets of $N(i)$ and
neglecting the non-overlap constraints on the sets $U_s$, we obtain
the bound}
\begin{equation}
\begin{aligned}
\sum_{W\subseteq V:\atop {i\in W,~|W|\geq 2}}
&\tree(G[W])\lambda^{|W|-1}
\leq \sum_{R\subseteq N(i)}\prod_{j\in R} \left( \sum_{U_j\subseteq
V\setminus\{i\} \atop U_j\ni j} \tree(G[U_j])\lambda^{|U_j|}\right)
\\
&= \sum_{R\subseteq N(i)}\prod_{j\in R} \left(\lambda+\lambda
\sum_{U_j\subseteq V\setminus\{i\} \atop {U_j\ni j\atop |U_j|\geq 2}}
\tree(G[U_j])\lambda^{|U_j|-1}\right)
\\
&\leq \sum_{R\subseteq N(i)}\prod_{j\in R} \left(\lambda+\lambda
be^{-b}\right) =\left(1+\lambda(1+ be^{-b})\right)^{|N(i)|}-1
\\
&\leq
e^{D\lambda (1+ be^{-b})} -1 =be^{-b}.
\end{aligned}
\end{equation}
\end{proof}

The above lemma gives a bound on the tails
\[
A_k=\sum_{m=1}^{\infty} \frac1{m!}\sum_{F_1,\dots,F_m\in
\ConIn(G)\atop \sum (|F_i|-1)\geq k} |\csp(L(F_1,\dots,F_m))|
\prod_{j=1}^m |z(F_j,\overline{H})|
\]
of the expansion \eqref{EQ:LNT-FZ}:

\begin{lemma}\label{LEM:TY-CONV}
Let $b>0$, let $G$ be a graph with maximum degree $D$, and let $H$ be
a weighted graph with  edgeweights in $[0,1]$. Then for every $k\ge
2$,
\[
A_k \leq b e^{-\eps k}|G|.
\]
\end{lemma}

\begin{proof}
Bounding $A_k$ by
\[
A_k\leq e^{-\eps k}\sum_{m=1}^{\infty}
\frac1{m!}\sum_{F_1,\dots,F_m\in \ConIn(G)\atop \sum (|F_i|-1) \geq
k} |\csp(L(F_1,\dots,F_m))| \prod_{j=1}^m |\tilde z_{F_j}|,
\]
we can ignore the condition on $\sum (|F_i|-1)$ to get
\[
A_k\leq e^{-\eps k}\sum_{m=1}^{\infty}
\frac1{m!}\sum_{F_1,\dots,F_m\in \ConIn(G)} |\csp(L(F_1,\dots,F_m))|
\prod_{j=1}^m |\tilde z_{F_j}|=e^{-\eps
k}\ln\stab(\GG,\hat{\mathbf{z}}),
\]
where $\hat{\mathbf{z}}$ denotes the vector
$(-|\tilde{z}_F|:~F\in\ConIn(G))$. We use Theorem \ref{thm:Dobrushin}
and Lemma~\ref{lem:conv-cond} to obtain the estimate
\begin{align*}
|A_k| & \leq e^{-\eps k}\sum_{F\in\ConIn(G)}\ln\left(1+ |\tilde z_F|
e^{b|F|}\right)\le e^{-\eps k}\sum_{i\in V}\sum_{F\in\ConIn(G)\atop
V(F)\ni i}\ln\left(1+ |\tilde z_F| e^{b|F|}\right)\\ &\leq e^{-\eps
k}\sum_{i\in V}\sum_{F\in\ConIn(G)\atop V(F)\ni i}|\tilde z_F|
e^{b|F|} \leq e^{-\eps k}|G| b.
\end{align*}
\end{proof}

\subsubsection{Proof of Theorem \ref{THM:HOM-CONV}}

We group the terms in Corollary \ref{COR:EXP1} according to the
subgraph of $G$ induced by the union of the $F_i$. More precisely,
for every graph $F$, define
\begin{equation}\label{VDEF}
v(F,H)=  \sum_{m=1}^\infty \frac1{m!}
\sum_{F_1,\dots,F_m\in\ConIn(F)\atop \cup_i V(F_i)=V(F)}
\csp(L(F_1,\dots,F_m)) \prod_{i=1}^m z(F_i,\overline{H}).
\end{equation}
We note that $v(H,F)=0$ if $F$ is disconnected, since then
$\csp(L(F_1,\dots,F_m))=0$. With this notation, we can also write
\eqref{EQ:LNT-FZ} as
\begin{equation}\label{EQ:LNT-V}
\ln t(G,H) =\sum_{F\in\ConIn(G)} v(F,H) =\sum_F \ind_0(F,G) v(F,H),
\end{equation}
where the last summation is extended over all isomorphism types of
connected graphs $F$ (clearly, graphs $F$ with more than $|G|$ nodes
contribute $0$).

Hence
\begin{equation}\label{EQ:LNT-SERIES}
\frac{\ln t(G_n,H)}{|G_n|} =\sum_F \frac{\ind_0(F,G_n)}{|G_n|}
v(F,H).
\end{equation}
Here $\ind_0(F,G_n)/|G_n|$ tends to some value as $n\to\infty$ by
left-convergence of the sequence $(G_n)$. Hence
\begin{align*}
\left|\frac{\ln t(G_n,H)}{|G_n|}-\frac{\ln t(G_m,H)}{|G_m|}\right|
&\le \sum_{|F|\le k}
\left|\frac{\ind_0(F,G_n)}{|G_n|}-\frac{\ind_0(F,G_m)}{|G_m|}\right|
|v(F,H)|\\
&~~~~~+\sum_{|F|>k}\Bigl(
\frac{\ind_0(F,G_n)}{|G_n|} +\frac{\ind_0(F,G_m)}{|G_m|}\Bigr) |v(F,H)|\\
&\le \sum_{|F|\le k}
\left|\frac{\ind_0(F,G_n)}{|G_n|}-\frac{\ind_0(F,G_m)}{|G_m|}\right|
|v(F,H)|+ 2 A_k.
\end{align*}
We can choose $k$ large enough so that the last term is less than
$\eps/2$, and then the first term will be less than $\eps/2$ if $n$
and $m$ are large enough.

This proves the theorem for $c(\bar H)<\frac 1{KD}$. We choose $b$ so
as to minimize $K$. For $b=2/5$ we get $K=7.964\dots<8$ (which is not
far from the best we get by this method).

\section{Right convergence implies left convergence}

\subsection{Linear independence of homomorphism functions}

The following lemmas extend some of the lemmas in \cite{ELS}.

\begin{lemma}\label{INJ-NONSING}
Let $F_1,\dots,F_k$ be non-isomorphic simple graphs. Then the
matrices
\[
M_\inj=\Bigl[\inj(F_i,F_j)\Bigr]_{i,j=1}^k
\]
and
\[
M_\surj=\Bigl[\surj(F_i,F_j)\Bigr]_{i,j=1}^k
\]
are nonsingular.
\end{lemma}

\begin{proof}
We may assume that the $F_i$ are ordered so that for $i<j$,
$|F_i|\le|F_j|$ and $|E(F_i)|\le|E(F_j)|$. Then the matrix $M_\inj$
is upper triangular and $M_\surj$ is lower triangular. Since both
matrices have positive diagonal entries, they are nonsingular.
\end{proof}

\begin{lemma}\label{HOM-NONSING-UNW}
Let $m\ge 1$ and let $\{F_1,\dots,F_k\}$ be a finite family of
non-isomorphic simple graphs closed under surjective homomorphic
image. Then the matrix
\[
M_\hom=\Bigl[\hom(F_i,F_j)\Bigr]_{i,j=1}^k
\]
is nonsingular.
\end{lemma}

\noindent Examples of such families are all simple graphs with at
most $q$ nodes, or all connected simple graphs with at most $q$
nodes, or with at most $m$ edges. We don't know if this proposition
holds for more general (perhaps all?) families of graphs.

\begin{proof}
We can express homomorphisms by surjective and injective
homomorphisms as follows:
\[
\hom(F_i,F_j)=\sum_{J} \frac{\surj(F_i,J)\inj(J,F_j)}{\aut(J)} ,
\]
where the summation extends over all simple graphs $J$ for which
$\surj(F_i,J)>0$. All such graphs $J$ belong to the family
$\{F_1,\dots,F_k\}$, which implies that if $M_\inj$ and $M_\surj$ are
as in Lemma \ref{INJ-NONSING}, and $D_\aut$ is the $k\times k$
diagonal matrix with the values $\aut(F_i)$ in the diagonal, then
\[
M_\hom= M_\surj D_\aut^{-1} M_\inj\, ,
\]
proving by Lemma \ref{INJ-NONSING} that $M_\hom$ is nonsingular.
\end{proof}

\subsection{Convergence of graph sequences}

\begin{theorem}\label{THM:RIGHT2LEFT}
Let $(G_1,G_2,\dots)$ be a sequence of simple graphs with degrees
bounded by $D$, and assume that there is a $\delta>0$ such that for
every simple looped graph $H$ with all degrees at least
$(1-\delta)|H|$, the sequence $\ln\hom(G_n,H)/|G_n|$ is convergent as
$n\to\infty$. Then the sequence $(G_1,G_2,\dots)$ is left-convergent.
\end{theorem}

\begin{proof}
Let $m\ge 1$ and let $\FF_m=\{F_1,\dots,F_N\}$ be the set of all
connected simple graphs with $2\le |F_i|\le m$. By Lemma
\ref{HOM-NONSING-UNW}, the matrix
\[
M=\Bigl[\hom(F_i,F_j)\Bigr]_{i,j=1}^N
\]
is nonsingular.

Let $q>m$. Add $q-|F_i|$ isolated nodes to $F_i$ and take the
complement to get a simple graph $H_i$ on $[q]$ with loops added at
the nodes. We think of $H_i$ as a weighted graph with all weights
$1$. Every node in $H_i$ has degree at least $q-m$, so if we choose
$q$ large enough, the condition on $H$ in the theorem is satisfied by
every $H_i$.

Consider any graph $G$ with all degrees at most $D$. We can rewrite
\eqref{EQ:LNT-FZ} as follows:
\begin{equation}\label{LOGT}
\ln t(G,H_i)=\sum_F \inj_0(F,G) u(F,H_i),
\end{equation}
where the summation extends over all connected graphs $F$, and
\begin{align}\label{WF2}
u(F,H_i)&=\sum_{k=1}^\infty \frac1{k!}\sum_{J_1,\dots,J_k\in
\Conn(F)\atop \cup J_i=F} \csp(L(J_1,\dots,J_k)) \prod_{r=1}^k
t(J_r,H_i{-}1).
\end{align}
Here
\[
t(J_r,H_i{-}1)= q^{-|J_r|}(-1)^{|E(J_r)|} \hom(J_r,F_i),
\]
and so
\[
\prod_{r=1}^k t(J_r,H_i{-}1) = (-1)^{\sum_r|E(J_r)|}q^{-\sum_r|J_r|}
\prod_{r=1}^k \hom(J_r,F_i).
\]
Note that the exponent of $q$ is less than $-|F|$ except for $k=1$
(when $J_1=F$). Hence
\begin{equation}\label{EQ:UFJ}
u(F_j,H_i)= q^{-|F_j|}(-1)^{|E(F_j)|}(\hom(F_j,F_i)+O(q^{-1})),
\end{equation}
for all $1\le i,j\le N$ and
\begin{equation}\label{EQ:UF}
u(F,H_i)=O(q^{-|F|-1})
\end{equation}
if $|F|>m$. Here and in what follows, the constants implied in the
$O$ may depend on $D$ and $m$ (and so also on $N$), but not on $q$,
$G$ and $\eps$.

By Lemma \ref{HOM-NONSING-UNW} it follows that if $q$ is large
enough, then the matrix $(u(F_i,H_j))_{i,j=1}^N$ is invertible.
Furthermore, if $(w_{ij})_{i,j=1}^n$ denotes its inverse, then
\[
w_{ij}=q^{|F_j|}(-1)^{|E(F_j)|}(M^{-1})_{ij} +O(q^{|F_j|-1}),
\]
and so
\begin{equation}\label{U-EST}
|w_{ij}|=O(q^{|F_j|})=O(q^m).
\end{equation}
Write
\begin{equation}\label{LOGT-APPROX}
\ln t(G,H_j)=\sum_{i=1}^N \inj_0(F_i,G) u(F_i,H_j) + R(G,H_j),
\end{equation}
where
\begin{equation}\label{LOGT-REM}
R(G,H_j)=\sum_{|F|> m} \inj_0(F,G) u(F,H_j)
\end{equation}
is a remainder term, which we can estimate as follows, using Lemma
\ref{LEM:SUBTREECOUNT}(d):
\begin{align}\label{R-EST}
|R(G,H_j)| & \le \sum_{r=m+1}^\infty \sum_{|F|=r} \inj_0(F,G)
|u(F,H_j)| = \sum_{r=m+1}^\infty \sum_{|F|=r} \inj_0(F,G)
O(q^{-r-1}) \nonumber\\
& = \sum_{r=m+1}^\infty 2^{Dr} |G| O(q^{-r-1}) =O(q^{-m-2})|G|.
\end{align}

We can view \eqref{LOGT-APPROX} as a system of $N$ equations in the
$N$ unknowns $\inj_0(F_i,G)$, from which these unknowns can be
expressed:
\begin{equation}\label{EQ:INJ0}
\inj_0(F,G) =\sum_{j=1}^N w_{ji} \ln t(G,H_j) +r_i(G),
\end{equation}
where
\[
r_i(G)=\sum_{j=1}^N w_{ji}R(G,H_j) =  O(q^m)O(q^{-m-2})|G| =
O(q^{-2}) |G|.
\]

Now let $\eps>0$ be given. Choosing $q$ large enough, we have
$|r_i(G)|<\eps|G|$ for all $1\le i\le N$ and every graph $G$. By
hypothesis, if $q$ is sufficiently large, then the sequence $\ln
(t(G_n,H_j))/{|G_n|}$ will be convergent for every $1\le j\le N$ as
$n\to\infty$, and so we can choose a positive integer $n_0$ such that
for $n,n'>n_0$, we have
\[
\left|\frac{\ln t(G_n,H_j)}{|G_n|}-\frac{\ln
t(G_{n'},H_j)}{|G_n|}\right|\le \eps q^{-m}.
\]
Then by \eqref{EQ:INJ0},
\begin{align*}
&\left| \frac{\inj_0(F_i,G_n)}{|G_n|}-\frac{\inj_0(F_i,G_{n'})}{|G_{n'}|} \right|\\
&= \sum_{i=1}^N \Bigl(w_{ji} \frac{\ln t(G_n,H_j)}{|G_n|}-w_{ji}
\frac{\ln t(G_{n'},H_j)}{|G_{n'}|}
+ \frac{r_i(G_n)}{|G_n|}-\frac{r_i(G_{n'})}{|G_{n'}|}\Bigr)\\
& \le \sum_{i=1}^N |w_{ji}|\cdot \Bigl|\frac{\ln
t(G_n,H_j)}{|G_n|}-\frac{\ln t(G_{n'},H_j)}{|G_{n'}|}\Bigr| + O(\eps) \\
&\le O(q^m) \eps q^{-m} + O(\eps) =O(\eps).
\end{align*}

This proves that the sequence $(\inj_0(F_i,G_n)/|G_n|:~n=1,2,\dots)$
is convergent for all $F_i\in\FF_m$. Since this holds for every $m\ge
1$, this proves the Theorem.
\end{proof}

In the proof above, the graphs $H_i$ have many twin nodes (all the
added nodes). We can always replace these by a single node of large
weight, to get a weighted graph $H_i'$ on at most $m+1$ nodes. The
argument would remain essentially the same if we replaced the $0$
edgeweights in $H_i$ by $1-\delta$ for any fixed $\delta>0$. Hence we
get the following variant:

\begin{theorem}\label{RIGHT2LEFT2}
Let $(G_1,G_2,\dots)$ be a sequence of simple graphs with degrees
bounded by $D$, and let $F$ be a simple graph. Assume that there is a
$\delta>0$ such that for every weighted graph $H$ on $|F|+1$ nodes
with all edgeweights in $[1-\delta,1]$, the sequence $(\ln
t(G_n,H))/|G_n|$ is convergent as $n\to\infty$. Then the sequence
$\hom(F,G_n)/|G_n|$ is convergent.
\end{theorem}

\section{Food for thought}

We mention some directions for further research.

\medskip

{\bf Quantitative bounds.} It would be interesting to make the
relationship between the numbers $\hom(F,G)/|G|$ and $(\ln
t(G,H))/|G|$ more explicit.

\medskip

{\bf Limit objects.} Benjamini and Schramm \cite{BSch} associated a
limit object with every left-convergent sequence of bounded degree
graphs, in the form of a probability distribution (with some special
properties) on countable rooted graphs with the same degree bound.
Other constructions of limit objects include graphings \cite{Elek1}
and measure preserving graphs \cite{LovH}. The ``left'' quantities
like $t(F,G)$ can be defined easily when $G$ is replaced by such a
limit object. Our Theorem \ref{THM:HOM-CONV} suggests that the
quantities $\ln t(G,H)/|G|$ can also be extended. However, the
definition (in other words, the description of the limiting value in
terms of the limit object) is not clear at all.

\medskip

{\bf Temperature.} Most of the time we have considered weighted
graphs $H$ with $\alpha_H=1$, whose edgeweights are between $0$ and
$1$, and close to $1$. Let us consider edgeweights of the form
$\beta_{ij}=\exp(-B_{ij})$, where $B_{ij}\ge 0$, and normalize so
that $\max_{i,j} B_{ij}=1$. Also consider the weighted graph
$H^{1/T}$, where $T$ is a parameter which in statistical physics
would be called the {\it temperature}, and the edge weights are
raised to the $1/T$ power. In this notation
\[
\hom(G,H^{1/T}) = \E_\phi \exp\Bigl(\frac1T \sum_{ij\in E(G)}
B_{\phi(i)\phi(j)}\Bigr),
\]
where $\phi$ is a random map $G\to H$. Furthermore,
\[
c(\overline{H}) =\max_i \sum_j \alpha_j (1-\beta_{ij}) \le \max_i
\frac1T \sum_j \alpha_j B_{ij} \le \frac1T.
\]
So it follows that if the temperature $T$ is larger than $2D$
then for
every left-convergent graph sequence $(G_n)$, the partition functions
$\hom(G_n,H^{1/T})$ are convergent.

What kind of convergence does it mean if the partition functions are
convergent at smaller temperature as well? This is not a local
property any more; still, is it related to some property ``from the
left''?

\medskip

{\bf Distance.} One would like to define a cut-distance type metric
for bounded degree graphs. Let $G_1$ and $G_2$ be two graphs with
degrees bounded by $D$ on the same node set $V=[n]$. Let $e_i(S,T)$
denote the number of edges in $G_i$ connecting $S$ and $T$
($S,T\subseteq V$). Suppose that we have
\begin{equation}\label{EQ:NEAR}
|e_1(S,T)-e_2(S,T)|\le \eps n
\end{equation}
for all $S,T\subseteq V$. Then for every weighted graph $H$ with
$V(H)=[q]$, $\alpha_H=1$ and $1/B \le \beta_{u,v}\le B$ for some
$B\ge 1$, we have
\begin{align*}
t(G_1,H) &= \E_\phi\Bigl(\prod_{ij\in E(G_1)}
\beta_{\phi(i)\phi(j)}\Bigr) = \E_\phi \Bigl(\prod_{u,v\in V(H)}
\beta_{uv}^{e_1(\phi^{-1}(u),\phi^{-1}(v))}\Bigr)\\
&\le \E_\phi \Bigl(\prod_{u,v\in V(H)}
\beta_{uv}^{e_2(\phi^{-1}(u),\phi^{-1}(v))}B^{\eps n\binom{q}{2}}\Bigr)\\
&= B^{\eps n\binom{q}{2}} t(G_2,H).
\end{align*}
Hence
\[
\left|\frac{\ln t(G_1,H)}{n}-\frac{\ln t(G_2,H)}{n}\right| \le \eps
\binom{q}{2} \ln B.
\]
By Theorem \ref{RIGHT2LEFT2}, this implies that $\hom(F,G_1)/|G_1|$
and $\hom(F,G_2)/|G_2|$ are also close if $n$ is large enough.

The trouble is that condition \eqref{EQ:NEAR} is too strong: it does
not hold for two random $D$-regular graphs, for example. Perhaps it
is possible to replace it by some condition asserting that it holds
``on the average''?

\section*{Appendix: Grids, a case study}

Let $P_n\square P_m$ denote the $n\times m$ grid (the Cartesian
product of a path with $n$ nodes and a path with $m$ nodes); let
$C_n\square P_m$ denote the $n\times m$ cylindrical grid (the
Cartesian product of a cycle with $n$ nodes and a path with $m$
nodes), and let $C_n\square C_m$ denote the $n\times m$ toroidal grid
(the Cartesian product of a cycle with $n$ nodes and a cycle with $m$
nodes). The sequences of $n\times m$ grids, cylindrical grids and
toroidal grids (and any merging of these three) are trivially
left-convergent if $m,n\to\infty$.

Grids are also right-convergent in a strong sense:

\begin{prop}\label{PROP:GRID}
For every weighted graph $H$, the sequence $\ln\hom(P_n\square
P_m)/nm$ is convergent as $n,m\to\infty$.
\end{prop}

\begin{proof}
Let
\[
s_0=\inf_{n,m}\frac{\ln\hom(P_n\square P_m)}{nm}.
\]
Fix an $\eps>0$, and choose $a,b\ge 1$ so that
\[
\frac{\ln\hom(P_a\square P_b)}{ab} \le s_0+\eps,
\]
and write $n=au+r$ and $m=bv+s$ where $0\le r<a$ and $0\le s<b$.
Using that trivially
\begin{equation}\label{SUB}
\hom(P_{n_1+n_2}\square P_m ,H)\le \hom(P_{n_1}\square
P_m,H)\hom(P_{n_2}\square P_m,H),
\end{equation}
it follows by an argument which could be called ``2-dimensional
Fekete Lemma'' that
\begin{align*}\label{}
&\frac{\ln\hom(P_n\square P_m)}{nm} \le \frac{u\ln\hom(P_a\square
P_m)+\ln\hom(P_r\square P_m)}{nm}\\
&~~~~~\le \frac{uv\ln\hom(P_a\square P_b)+u\ln\hom(P_a\square
P_s)+v\ln\hom(P_r\square P_b)+\ln\hom(P_r\square P_s)}{nm}\\
&~~~~~\le s_0+\eps +O(\frac1n +\frac1m)
\end{align*}
if $n,m\to\infty$.
\end{proof}

The situation is more complicated with cylindrical grids:

\begin{prop}\label{PROP:CYLGRID}
{\rm (a)} If $n$ is restricted to even numbers, then for every
weighted graph $H$, the sequence $\ln\hom(C_n\square P_m)/nm$ is
convergent as $n,m\to\infty$.

\smallskip

{\rm (b)} If $H$ is connected and nonbipartite, then the sequence
$\ln\hom(C_n\square P_m)/nm$ is convergent as $n,m\to\infty$.
\end{prop}

\begin{proof}
(a) We may assume that the edgeweights of $H$ are in $[0,1]$. Then
\[
\hom(C_n\square P_m)\le \hom(P_n\square P_m),
\]
and hence
\[
\limsup_{n,m\to\infty}\frac{\ln\hom(C_n\square P_m)}{nm}
\le\limsup_{n,m\to\infty} \frac{\ln\hom(P_n\square P_m)}{nm} \le s_0.
\]
On the other hand, consider the subsets of nodes
$A_1=\{(0,x):~x=1\dots m\}$ and $A_1=\{(n/2,x):~x=1\dots m\}$ in
$L'_{n,m}$. We think of $L'_{n,m}$ as two $(n/2+1)\times m$ grids
$B_1$ and $B_2$, glued together along $A_1$ and $A_2$. For any fixed
map $\sigma:~A_1\cup A_2\to V(H)$, let $M_\sigma$ denote the number
of $H$-colorings of $B_1$ extending $\sigma$. Then the number of
$H$-colorings of $L'_{n,m}$ extending $\sigma$ is $M_\sigma^2$, and
so
\[
\hom(C_n\square P_m,H)=\sum_\sigma M_\sigma^2,
\]
while
\[
\hom(P_{n/2+1}\square P_m,H)=\sum_\sigma M_\sigma.
\]
Thus by Cauchy-Schwartz,
\begin{align*}
\hom(C_n\square P_m,H)&=\sum_\sigma M_\sigma^2 \ge
\frac{1}{|V(H)|^{2m}}\left(\sum_\sigma M_\sigma\right)^2 \\
&=\frac{1}{|V(H)|^{2m}}\hom(P_{n/2+1}\square P_m,H)^2,
\end{align*}
and so
\begin{align*}
\hom(C_n\square P_m,H) &\ge
\frac{1}{nm}\left(2\ln\hom(P_{n/2+1}\square P_m,H)-2m \ln
|V(H)|\right) \\
&= (1+\frac{2}{m})\frac{\ln\hom(P_{n/2+1}\square P_m,H)}{(n/2+1)m} -\frac{2}{n} \\
&\ge (1+\frac{2}{m})s_0-\frac{2}{n}.
\end{align*}
This lower bound tends to $s_0$ as $n,m\to\infty$.

\medskip

(b) As before, it is trivial that
\[
\limsup_{n,m\to\infty} \frac{\ln\hom(C_n\square P_m)}{nm} \le s_0,
\]
so our task is to estimate $\hom(C_n\square P_m)$ from below.

Fix any $\eps>0$, and then fix an $n_0>0$ so that if $n,m\ge n_0$
then $\ln\hom(P_n\square P_m)/(nm)\ge s_0-\eps$.

For a given $m\ge n_0$, we construct an auxiliary weighted graph $G$
as follows. The nodes of $G$ are all maps $V(P_m)\to V(H)$ with
positive weight, where the weight of a map $\phi$ is $\prod_{i\in
V(P_m)} \alpha_{\phi(i)}(H)\prod_{ij\in E(P_m)}
\beta_{\phi(i)\phi(j)}(H)$. Two maps $\phi,\psi$ are connected by an
edge with weight $b_{\phi,\psi}=\prod_{i\in V(P_m)}
\beta_{\phi(i)\psi(i)}$ if this weight is positive.

\smallskip

{\bf Claim.} No connected component of $G$ is bipartite.

\smallskip

Consider any node $x$ of $G$; we want to show that there is a walk in
$G$ of odd length starting and ending at $x$.

The node $x$ is a map of $P_m$ into $H$, which can also be viewed as
a walk $W$ in $H$ with $m$ nodes. Since $H$ is connected and
nonbipartite, we can extend $W$ to a closed walk
$W'=(v_1,v_2,\dots,v_p)$ of odd length $p$ (so $W=(v_1,\dots,v_m)$).
Let $W_i=(v_i,v_{i+1},\dots,v_{i+m-1})$ (where the indices are taken
modulo $p$). The $W_i$ corresponds to a node in $H$ and
$(W_1,\dots,W_p)$ is a walk in $H$ of odd length containing $W$. This
proves the Claim.

\smallskip

Now observe that
\[
\hom(P_n\square P_m,H)= \hom(P_n,G)
\]
and
\[
\hom(C_n\square P_m,H)= \hom(C_n,G).
\]
Since $G$ is connected and nonbipartite, it is easy to show that
$\ln\hom(P_n\square P_m)/(nm)$ and $\ln\hom(C_n\square P_m)/(nm)$
tend to the same value as $n\to\infty$, and we know that this value
is at least $s_0-\eps$. So for a fixed $m\ge n_0$, if $n\ge
n_0(m,\eps)$, then $\ln\hom(P_n\square P_m)/(nm)\ge s_0-2\eps$.
Choose
\[
m_0=\max\Big\{n_0, \frac{2\ln q}{\eps}\Bigr\},
\]
and let
\[
N_0=\max\{\frac{s_0}{\eps}m_0, n_0(m_0,\eps)\}.
\]
Assume that $n,m\ge N_0$. Consider the cylindrical grid $C_n\square
P_{m_0}$. Since $n\ge n_0(m_0,\eps)$, this grid has at least
$\exp((s_0-2\eps)nm_0)$ $H$-colorings. Hence there is a way to fix
the map on the two ``boundary'' cycles so that after fixing it, we
still have at least
\[
\frac{\exp((s_0-2\eps)nm_0)}{q^{2n}} \ge\exp((s_0-3\eps)nm_0)
\]
extensions. Let $\alpha_1$ and $\alpha_2$ denote these maps of the
two boundary cycles.

Now consider the cylindrical grid $C_n\square P_{m_0}$. Map the
first, $m_0$-th, $(2m_0-1)$-st etc.\ $n$-cycles alternatingly
according to $\alpha_1$ and $\alpha_2$. Then the number of ways to
extend this to the layer between two such cycles is at least
$\exp((s-3\eps)nm_0)$, so the number of ways to extend the mapping to
the whole graph is at least
\[
\exp\Bigl((s_0-3\eps)nm_0\bigl\lfloor \frac{m}{m_0}\bigr\rfloor\Bigr)
\ge \exp\Bigl((s_0-3\eps)nm_0\bigl(\frac{m}{m_0}-1\bigr)\Bigr) \ge
\exp((s_0-4\eps)nm).
\]
Hence
\[
\ln\hom(C_n\square P_m)/(nm) \ge s_0-4\eps.
\]
\end{proof}

For toroidal grids, the situation is similar (and so are the proofs),
and we only state the result:

\begin{prop}\label{PROP:TOROIDAL}
{\rm (a)} If $n$ and $m$ are restricted to even numbers, then for
every weighted graph $H$, the sequence $\ln\hom(C_n\square C_m)/(nm)$
is convergent as $n,m\to\infty$.

\smallskip

{\rm (b)} If $H$ is connected and nonbipartite, then the sequence
$\ln\hom(C_n\square C_m)/(nm)$ is convergent as
$n,m\to\infty$.\proofend
\end{prop}
\end{document}